\documentclass[11pt, oneside, reqno]{amsart}   	% use "amsart" instead of "article" for AMSLaTeX format
\usepackage{geometry}                		% See geometry.pdf to learn the layout options. There are lots.
\usepackage{graphicx}	

\usepackage{amssymb}
\usepackage{amsmath}
\usepackage{amsthm}

\usepackage{lmodern}

\usepackage{mathrsfs}

\usepackage{setspace}
\usepackage{changepage}

\usepackage{xcolor}

\usepackage{tikz}
\usetikzlibrary{matrix,arrows,chains,positioning,scopes}
\tikzset{>=stealth}
\usepackage{tikz-cd}

\usepackage[labelfont=bf,textfont=it]{caption}

\newcommand{\R}{\mathbb{R}} 
\newcommand{\Z}{\mathbb{Z}}

\newcommand{\bbM}{\mathbb{M}}

\newcommand{\scrH}{\mathscr{H}}

\newcommand{\quand}{\quad \text{and} \quad}

\DeclareMathOperator{\diam}{diam} 
\DeclareMathOperator{\im}{im}

\newtheorem{thm}{Theorem}[section]
\newtheorem{prop}[thm]{Proposition}
\newtheorem{lem}[thm]{Lemma} 
\newtheorem{cor}[thm]{Corollary}

\theoremstyle{definition}
\newtheorem{defn}[thm]{Definition}

\newtheorem{ex}[thm]{Example}
\newtheorem*{rem}{Remark}

\title{Equidistant Sets on Alexandrov Surfaces}

\date{\today}

\author{Logan S. Fox} 
\address{Fariborz Maseeh Dept. of Math. \& Stat., Portland State University} 
\email{logfox@pdx.edu}

\author{J.J.P. Veerman}
\address{Fariborz Maseeh Dept. of Math. \& Stat., Portland State University}
\email{veerman@pdx.edu}

\begin{document} % BEGIN % 

\begin{abstract} 
We examine properties of equidistant sets determined by nonempty disjoint compact subsets of a compact 2-dimensional Alexandrov space (of curvature bounded below). The work here generalizes many of the known results for equidistant sets determined by two distinct points on a compact Riemannian 2-manifold. Notably, we find that the equidistant set is always a finite simplicial 1-complex. These results are applied to answer an open question concerning the Hausdorff dimension of equidistant sets in the Euclidean plane. 
\end{abstract} 

\maketitle

%%%%%%%%%%%%%%%%%%%%%%%%%%%%

\section{Introduction} %%% INTRODUCTON %%%

%%%%%%%%%%%%%%%%%%%%%%%%%%%%

Given two nonempty sets, say \(A\) and \(B\), in a metric space \(X\), we will denote by \(E(A,B)\) the set of points whose distance to \(A\) is the same as their distance to \(B\). 
That is, \(E(A,B)\) is the set of points \emph{equidistant} to \(A\) and \(B\). 
Perhaps unsurprisingly, the structure of the set \(E(A,B)\) depends not only on qualities of the sets \(A\) and \(B\), but also the properties of the ambient space \(X\). 
The goal of this paper is to extend work done in \cite{bv2006}, \cite{bv2007}, and \cite{herreros}, where many results are given for equidistant sets determined be pairs of points on compact Riemannian 2-manifolds. 
We generalize many of these results in two directions: (i) we consider equidistant sets determined by nonempty disjoint compact subsets; and (ii) the ambient space is a compact Alexandrov surface. 

Our work is aided in part by the observation that every point of the equidistant set (determined by \(A\) and \(B\)) admits at least two distinct shortest paths to \(A\cup B\) (one to \(A\) and one to \(B\)). As such, the equidistant set is contained in the \emph{cut locus} to \(A\cup B\). This allows us to make use of results concerning the cut locus to a compact set on an Alexandrov surface -- particluarly, those developed by Shiohama and Tanaka in \cite{ST1996}. Although equidistant sets are closely related to cut loci, we find that equidistant sets are in general much `nicer' (see Corollary \ref{cor: finite measure} and its preceding paragraph). 

This paper is organized as follows: In Section \ref{sec: prelim}, we review important properties of length spaces and Alexandrov spaces. 
In Section \ref{sec: equidistant}, we prove some basic facts concerning equidistant sets, particularly when \(X\) is a proper length space with no branching geodesics. 
Section \ref{sec: ES CBB} combines techniques from \cite{ST1996}, \cite{bv2006}, and \cite{bv2007}, to characterize equidistant sets on compact Alexandrov surfaces. 
Finally, in Section \ref{sec: the plane} we address the two open questions found in \cite{ponce} concerning equidistant sets in the Euclidean plane.

%%%%%%%%%%%%%%%%%%%%%%%%%%%%%%%%%

\section{Preliminaries} %%% PRELIMINARIES %%%
\label{sec: prelim}

%%%%%%%%%%%%%%%%%%%%%%%%%%%%%%%%%

This section collects necessary definitions and results concerning the fundamentals of Alexandrov spaces. With the exception of the doubling theorem, the content of this section can be found in the standard survey articles \cite{bgp92}, \cite{plaut}, and \cite{shiohama}, as well as the text \cite{bbi2001}.

\subsection{Length Spaces} 

A \emph{path} (or \emph{curve}) in a metric space \(X\) is a homeomorphism \(\gamma: [a,b] \to X\), where \([a,b]\) is an interval of \(\R\) (in other words, by path or curve we mean a Jordan arc). The \emph{length} of any path \(\gamma\) is denoted \(L(\gamma)\) and is the supremum of the distance along finite partitions of the path: 
\[ L(\gamma) = \sup \left\{ \sum_{k=1}^{n-1} d\big( \gamma(t_{k}), \gamma(t_{k+1}) \big) : a=t_1 < t_2<\cdots < t_n = b \right\} . \] 
If the length is finite, then the path is said to be \emph{rectifiable}.

A path \(\gamma\) is a \emph{shortest path} if \(d(\gamma(s),\gamma(t)) = |s-t|\) for all \(s,t\in [a,b]\). 
Note that shortest paths are distance minimizing and parametrized by arc length. 
A path is a \emph{geodesic} if it is locally a shortest path. 
Typically, we take the domain of a shortest path \(\gamma\) to be \([0,T]\) so that \(t=d \big( \gamma(0),\gamma(t) \big)\) for all \(t\in [0,T]\). 
However, it can be convenient to reparametrize a path to some other domain. 
Therefore, we say that a shortest path \(\gamma: [c,d] \to X\) is \emph{linearly parametrized} if \(d(\gamma(s),\gamma(t)) = \lambda |s-t|\) for all \(s,t\in [c,d]\), where \(\lambda\) is the constant 
\[ \lambda = \frac{d(\gamma(c),\gamma(d))}{d-c} . \]

A \emph{length space} is a metric space such that the distance between any pair of points is equal to the infimum of lengths of paths connecting those points. 
In an arbitrary length space, there may be no shortest path connecting a given pair of points; however, as longs as the space is complete and locally compact, the Hopf--Rinow theorem guarantees a shortest path always exists.

\begin{thm}[Hopf--Rinow]\label{hopf-rinow}
If \(X\) is a complete and locally compact length space, then every closed and bounded subset of \(X\) is compact, and any two points in \(X\) can be connected by a shortest path.
\end{thm}

In addition to the Hopf--Rinow theorem, we will also make use of a well-known consequence of the Arzel\'a--Ascoli theorem, which tells us that every sequence of shortest paths in a compact length space contains a convergent subsequence.

\begin{thm}[Arzel\`a--Ascoli] 
If \(X\) is a separable metric space and \(Y\) is a compact metric space, then every uniformly equicontinuous sequence of functions \(f_n: X\to Y\) has a subsequence that converges uniformly to a continuous function \(f:X\to Y\). 
\end{thm}

\begin{cor}\label{cor: arzela} 
If \(X\) is a compact length space and \(\{\gamma_n\}_n\) is a sequence of linearly parametrized shortest paths, \(\gamma_n : [0,1]\to X\), then \(\{\gamma_n\}_n\) contains a subsequence which converges uniformly to a linearly parametrized shortest path. 
\end{cor}

%%% ALEXANDROV SPACES %%%

\subsection{Alexandrov Spaces} 

We will denote by \(\bbM^2_k\) the complete \(2\)-dimensional Riemannian manifold of constant sectional curvature \(k\). 
Letting \(X\) be a length space and \(x,y,z\in X\) be three distinct points satisfying 
\[ d(x,y) + d(x,z) + d(y,z) \leq 2 \diam \bbM^2_k = \begin{cases} 2\pi / \sqrt{k} & \text{ if } k > 0 \\ \infty & \text{ if } k \leq 0 \end{cases} \] 
we can isometrically embed these points as the vertices of a geodesic triangle in \(\bbM^2_k\), which we will denote \(\Delta^k (x,y,z)\). 
We then denote by \(\angle^k(y,z)\) the angle at \(x\) in \(\Delta^k(x,y,z)\). 

Given two shortest paths emanating from the same point, say \(\gamma: [0,T] \to X\) and \(\eta:[0,S]\to X\) with \(\gamma(0) = \eta(0)\), the \emph{upper angle} between \(\gamma\) and \(\eta\) is defined by fixing a \(k\) and taking the limit superior of angles \(\angle^k \big( \gamma(t),\eta(s) \big)\) as \(s\) and \(t\) vanish, 
\begin{equation}\label{eq: angle} 
\angle^+(\gamma,\eta) = \limsup_{s,t\to 0^+} \angle^k \big( \gamma(t),\eta(s) \big) . 
\end{equation} 
One similarly defines the lower angle by taking the limit inferior in \eqref{eq: angle}. If the upper and lower angles are the same, then we say that the angle exists and simply denote it \(\angle(\gamma, \eta)\). 

If there is a \(k\) such that for any \(s,t>0\) and any shortest paths \(\gamma\) and \(\eta\) as above, we find \(\angle^+(\gamma , \eta) \geq \angle^k\big( \gamma(t),\eta(s) \big)\), then we say that \(X\) is a space of \emph{curvature bounded below} (or curvature \(\geq k\)).\footnote{On the other hand, if \(\angle^+(\gamma , \eta) \leq \angle^k\big( \gamma(t),\eta(s) \big)\), then we say that \(X\) is a space of curvature bounded above, but we will not consider such a space in this article.} 
It is well known that the bounded curvature condition guarantees the existence of angles between shortest paths. 

For the purposes of this article, we define Alexandrov spaces as follows. 

\begin{defn}\label{def: alexandrov}
An \emph{Alexandrov space} is a complete and locally compact length space of curvature bounded below. 
\end{defn}

In general it is not required that Alexandrov spaces are locally compact. However, we will rely heavily on the Hopf--Rinow theorem, so we define Alexandrov spaces accordingly. In fact, the later sections further assume that the spaces in question are \(2\)-dimensional, which we will call \emph{Alexandrov surfaces}. 

When referring to the dimension of an Alexandrov space, we generally mean the Hausdorff dimension (although this is known to be equivalent to topological dimension \cite[Theorem 156]{plaut}). 
See \cite[Chapter 10]{bbi2001} -- the source of each item of the following proposition -- for a more complete discussion.

\begin{prop}[\cite{bbi2001}]\label{prop: dimension}
Let \(X\) be an Alexandrov space. \\ 
(i) The Hausdorff dimension of \(X\) is a nonnegative integer or infinity. \\ 
(ii) All open subsets of \(X\) have the same Hausdorff dimension. \\ 
(iii) If \(X\) is \(2\)-dimensional, then it is a topological \(2\)-manifold, possibly with boundary. 
\end{prop}

%%% NO BRANCHING GEODESICS %%% 

\subsection{No Branching Geodesics} 

One important consequence of the lower curvature bound condition, is that it removes the possibility of branching geodesics. 
Intuitively, two geodesics which emanate from the same point are said to \emph{branch} if they initially overlap for some time interval, but then become disjoint at some point on the interior of both paths. 

\begin{defn}\label{defn: branch}
Let \(X\) be a length space and let \(\gamma,\eta:[0,1]\to X\) be linearly parametrized shortest paths with \(\gamma(0) = \eta(0)\). 
If there are values \(t,s\in (0,1)\) such that 
\[ \gamma\big( [0,t] \big) = \eta\big( [0,s] \big) \quad \text{and} \quad \gamma\big([t,t+\varepsilon]\big) \cap \eta\big([s,s+\varepsilon]\big) = \gamma(t) \] 
for some \(\varepsilon>0\), then \(\gamma\) and \(\eta\) are said to \emph{branch}. The point \(\gamma(t)\) is the \emph{branch point} between \(\gamma\) and \(\eta\). 
\end{defn}

\begin{ex}
Consider a length space formed by two cones with their vertices identified. Any pair of shortest paths which begin at the same point on one cone and end at distinct equidistant points on the other cone will branch at the shared vertex. 
\end{ex}

\begin{lem}[{\cite[Lemma 2.4]{shiohama}}]\label{lem: no branch}
If \(X\) is an Alexandrov space, then there are no branching geodesics in \(X\). 
\end{lem}

\begin{cor}\label{cor: subpath} 
Let \(X\) be an Alexandrov space. For any shortest paths \(\gamma : [0,T] \to X\) and \(\eta : [0,S] \to X\), with \(\gamma(0) = \eta(0)\), \(\angle(\gamma,\eta) = 0\) if and only if one path is a subpath of the other. 
\end{cor}

Note that the converse of Lemma \ref{lem: no branch} is not necessarily true, even for compact length spaces, as the following example illustrates.

\begin{ex} 
Let \(X\) be the surface of revolution in \(\R^3\) obtained by rotation the graph of \(z=\sqrt{x}\), for \(0 \leq x \leq 1\), around the \(z\)-axis. 
In other words, 
\[ X = \left\{ (x,y,z)\in \R^3 : z = \sqrt[4]{x^2 + y^2} \ , \ x^2 + y^2 \leq 1 \right\} . \] 
Since \(X\) is the continuous image of a compact set, it is compact. First equipping \(X\) with the subspace metric induced by the Euclidean norm, we can then equip \(X\) with an intrinsic metric by defining \(d(x,y)\) to be the infimum of lengths of paths from \(x\) to \(y\). This space does not have branching geodesics (and no geodesic contains the singular point in its interior); however, there is no lower curvature bound, which we can see by examining a triangle with one vertex at the singular point \((0,0,0)\). 
\end{ex}

%%% THE SPACE OF DIRECTIONS %%%

\subsection{The Space of Directions}

Let \(X\) be an Alexandrov space. For any fixed \(x\in X\), let \(\Gamma_x\) be the set of all shortest paths emanating from \(x\). 
We define an equivalence relation on \(\Gamma_x\) by 
\[ \gamma \sim \eta \iff \angle (\gamma , \eta) = 0 \] 
and let \(\Sigma_x\) be the set of equivalence classes of \(\Gamma_x / \sim\). 
It is straightforward to verify that the angle is a metric on \(\Sigma_x\). 
The \emph{space of directions}, \(S_xX\), is then defined as the completion of \((\Sigma_x , \angle)\). 
For any shortest path \(\gamma\), we will denote by \([\gamma]\) (or simply \(\gamma\) when there is no ambiguity) its equivalence class in \(S_xX\).

If \(X\) is an \(n\)-dimensional Alexandrov space (with \(n\geq 2\)) then for every \(x\in X\), the space of directions \(S_xX\) is itself a compact \((n-1)\)-dimensional Alexandrov space of curvature \(\geq 1\) \cite[Theorem 10.8.6]{bbi2001}. 
For \(2\)-dimensional Alexandrov spaces, this tells us that the space of directions at every point is either a line segment or a circle. Furthermore, by the radius sphere theorem \cite{GP1993}, \(\diam S_xX \leq \pi\).

The \emph{tangent cone} (or tangent space) at \(x\), which we denote \(T_xX\), is the Euclidean cone over \(S_xX\): 
\[ T_xX = \big( S_xX \times [0,\infty) \big) / \big( S_xX \times \{0\} \big) . \] 
We equip the tangent cone with the metric induced by the law of cosines, 
\[ d \big( (\gamma,t) , (\eta,s) \big) = \sqrt{ t^2 + s^2 - 2st\cos\big( \angle(\gamma , \eta) \big) } . \] 
When \(X\) is an Alexandrov space of curvature \(\geq k\), \(T_xX\) is an Alexandrov space of curvature \(\geq 0\) for all \(x\in X\). 
Furthermore, if \(X\) is \(n\)-dimensional, then so is \(T_xX\) (this follows from the dimension of \(S_xX\)). 

There is a notion of exponential map, \(\exp_x : T_xX \to X\), defined simply by 
\[ \exp_x (\gamma,t) = \gamma(t) . \] 
However, this assumes there is in fact a shortest path \(\gamma\) at \(x\) with direction \([\gamma]\) and \(L(\gamma)\geq t\). 
In general, there may be no neighborhood of the origin such that \(\exp_x\) is defined for all points in the neighborhood (see Example \ref{ex: flat sphere}). We will make limited use of the exponential map in Section \ref{sec: ES CBB}.

%%% THE DOUBLING THEOREM %%%

\subsection{The Doubling Theorem} 
\label{subsec: doubling}

Let \(X\) be an \(n\)-dimensional Alexandrov space with nonempty boundary \(\partial X\), and let \(\phi: X \to Y\) be an isometry. We define the doubling of \(X\), which we denote \(\widetilde{X}\), as the gluing of \(X\) with itself (i.e. with \(Y\)) along the boundary, 
\[ \widetilde{X} = X \cup_{\phi |_{\partial X}} Y . \] 
After identifying \(\partial X\) with \(\phi(\partial X)\), we can equip the doubled space with the metric 
\[ \tilde{d}(x,y) = \begin{cases} 
d_X(x,y) & \text{ if } x,y\in X \\ 
d_Y(x,y) & \text{ if } x,y\in Y \\ 
\inf \{ d(x,z) + d(y,z) : z\in \partial X \} & \text{ if } x\in X \text{ and } y\in Y . 
\end{cases} \] 
It is straightforward to see that \(\widetilde{X}\) is a length space, but it is in fact also an \(n\)-dimensional Alexandrov space of curvature \(\geq k\). 

\begin{thm}[Doubling Theorem \cite{perelman}]\label{thm: doubling} 
Given an \(n\)-dimensional Alexandrov space \(X\) of curvature \(\geq k\) (for some \(k\in\R\)) with nonempty boundary, the doubled space \(\widetilde{X}\) is an \(n\)-dimensional Alexandrov space of curvature \(\geq k\) with empty boundary. 
\end{thm} 

Note that Perelman's proof of the doubling theorem was never formally published.\footnote{Anton Petrunin kindly keeps a copy of Perelman's preprint available on his website: \texttt{https://anton-petrunin.github.io/papers/}} 
However, it is a direct consequence of the gluing theorem \cite{petrunin}, which generalizes the doubling theorem to any two Alexandrov spaces with isometric boundaries.

%%%%%%%%%%%%%%%%%%%%%%%%%%%%%%%%

\section{Equidistant Sets} %%% EQUIDISTANT SETS %%% 
\label{sec: equidistant} 

%%%%%%%%%%%%%%%%%%%%%%%%%%%%%%%%

\begin{defn}
Let \((X,d)\) be a metric space. Given two nonempty sets \(A,B \subseteq X\), the \emph{equidistant set} (also \emph{mediatrix} or \emph{midset}) determined by \(A\) and \(B\) is the set of points of equal distance to \(A\) and \(B\); 
\[E(A,B) = \{x\in X : d(x,A) = d(x,B) \} . \] 
In the case of singleton sets, say \(\{a\}\) and \(\{b\}\), we use \(E(a,b) = E(\{a\},\{b\})\) for simplicity of notation. 
\end{defn}

Although equidistant sets are defined for any metric space, properties of equidistant sets can vary greatly depending on the metric.

\begin{ex}
Consider the real line with the following metrics 
\[ d_1(x,y) = \frac{|x-y|}{1+|x-y|} \quad \text{and} \quad d_2(x,y) = \begin{cases} |x-y| , & \text{ if } |x-y| < 1 \\ 1 , & \text{ otherwise.} \end{cases} \]
The metrics \(d_1\) and \(d_2\) are strongly equivalent with \(d_1(x,y) \leq d_2(x,y) \leq 2d_1(x,y)\). 
It is easy to verify that for any \(p,q\in (\R,d_1)\), we have \(E(p,q) = \{\frac{p+q}{2}\}\). However, in \((\R,d_2)\), if \(p=2\) and \(q=-2\), 
\[ E(p,q) = (-\infty,-3] \cup [ -1,1] \cup [3,\infty) . \] 
\end{ex}

It should be further observed that one can easily construct a metric space for which the equidistant set determined by two points is actually empty. 
However, this is easily avoided by assuming the space is path connected (which length spaces always are).

\begin{lem} 
Let \(X\) be a path connected metric space. If \(A\) and \(B\) are nonempty subsets of \(X\), then \(E(A,B)\) is nonempty. 
\end{lem} 

\begin{proof} 
Define the function \(f_{AB}: X \to \R\) by 
\[ f_{AB}(x) = d(x,A) - d(x,B) \] 
and notice that \(E(A,B) = f^{-1}_{AB}( \{0\})\). 
Fixing some \(a\in A\) and \(b\in B\) and letting \(\sigma: [0,1]\to X\) be a path from \(a\) to \(b\), the intermediate value theorem tells us that \(f_{AB}(\sigma(t)) = 0\) for some \(t\in [0,1]\). 
\end{proof}

In order to ensure that our equidistant sets are `well behaved' in some sense, we will follow \cite{bv2006} and assume the ambient metric space is always a proper length space with no branching geodesics. In particular, this includes Alexandrov spaces.

\begin{defn} 
Given a point \(x\) and a set \(A\), the \emph{metric projection} \(P_A(x)\) is the set of points of \(A\) which realize the distance \(d(x,A)\), 
\[ P_A(x) = \{a\in A : d(x,a) = d(x,A)\} . \] 
Whenever we say that \(\gamma : [0,T] \to X\) is a shortest path from a point \(x\) to a set \(A\), we mean that \(\gamma(0) = x\) and \(\gamma(T) \in P_A(x)\), assuming \(P_A(x)\) is nonempty. 
\end{defn}

The following proposition, and its corollaries, extend properties of equidistant sets determined by distinct points found in \cite[\S 2]{bv2006} to equidistant sets determined by disjoint closed sets.

\begin{prop}\label{prop: ES basics} 
Let \(X\) be a proper length space with no branching geodesics, \(A\) and \(B\) be nonempty disjoint closed subsets of \(X\), and \(x\in E(A,B)\) be given. \\ 
(i) If \(\gamma : [0,T] \to X\) is a shortest path from \(x\) to \(A\), then \(d(\gamma(t), A) < d(\gamma(t),B)\) for all \(t\in (0,T]\). \\ 
(ii) There is a path connecting some \(a\in A\) to some \(b\in B\) which does not intersect \(E(A,B)\setminus\{x\}\). 
\end{prop}

\begin{proof} 
For (i): Let \(x\in E(A,B)\) and \(\gamma\) with \(\gamma(T) = a \in P_A(x)\) be given. 
For the sake of contradiction, suppose that there is a \(t_0\in (0,T)\) such that \(d(\gamma(t_0), B) \leq d(\gamma(t_0),A)\). Fix \(b_0\in P_B(\gamma(t_0))\). Since \(d(x,a) \leq d(x,b_0)\) and \(d(\gamma(t_0),b_0) \leq d(\gamma(t_0),a)\), 
\begin{align*} 
d(x,\gamma(t_0)) + d(\gamma(t_0),b_0) & = d(x,a) - d(\gamma(t_0), a) + d(\gamma(t_0),b_0)
\\ & \leq d(x,b_0) - d(\gamma(t_0), a) + d(\gamma(t_0),a) 
\\ & = d(x,b_0) 
\end{align*} 
so \(d(x,b_0) = d(x,\gamma(t_0)) + d(\gamma(t_0),b_0)\). Given that \(a \neq b_0\), this implies that \(\gamma\) is branching, which contradicts our hypothesis. Thus, \(d(\gamma(t),A) < d(\gamma(t),B)\) for all \(t\in(0,T]\). 

For (ii): Given any \(x\in E(A,B)\), if \(\gamma\) if a shortest path from \(a\in P_A(x)\) to \(x\), and \(\eta\) is a shortest path from \(x\) to \(b\in P_B(x)\), then by part (i), the concatenation of \(\gamma\) and \(\eta\) connects \(a\) to \(b\) without intersecting \(E(A,B)\setminus \{x\}\). 
\end{proof}

\begin{cor}\label{cor: no interior}
If \(X\) is a proper length space with no branching geodesics, then for any nonempty disjoint closed sets \(A,B\subseteq X\), the equidistant set \(E(A,B)\) has empty interior. 
\end{cor} 

\begin{proof} 
Let \(x\in E(A,B)\) be given. By Proposition \ref{prop: ES basics}, an open ball of any radius centered at \(x\) cannot be contained in \(E(A,B)\) since any shortest path from \(x\) to \(A\) (or \(B\)) immediately leaves \(E(A,B)\).  
\end{proof}

Our next corollary is related to idea of \emph{minimal separating} introduced in the study of Brillouin spaces and equidistant sets determined by points \cite{vprs}. 
In our case, \(E(A,B)\) separates \(X\) into two sets 
\[ \{x\in X : d(x,A) < d(x,B) \} \quand \{x\in X : d(x,B) < d(x,A)\} . \] 
but each of these sets may consist of more than one component of \(X\setminus E(A,B)\). 
However, if we additionally assume that \(A\) and \(B\) are connected, then this separation is minimal.

\begin{defn} 
Let \(X\) be a connected metric space. A set \(E\subseteq X\) is \emph{separating} if \(X\setminus E\) consists of more than one component. If \(E\) is separating, but no proper subset of \(E\) is separating, then \(E\) is \emph{minimal separating}. 
\end{defn}

\begin{cor}\label{cor: minimal sep}
Let \(X\) be a proper length space with no branching geodesics. If \(A,B\subseteq X\) are nonempty disjoint closed and connected, then \(E(A,B)\) is minimal separating. 
\end{cor} 

\begin{proof} 
For simplicity of notation, define the sets 
\[ X_A = \{x\in X : d(x,A) < d(x,B) \} \quand X_B = \{x\in X : d(x,B) < d(x,A)\} \] 
so that \(X\) is the union of the disjoint sets \(X_A\), \(X_B\), and \(E(A,B)\). 

First, we will show that \(X_A\) is in fact a single component of \(X\setminus E(A,B)\). 
Certainly \(A\subseteq X_A\), and given that \(A\) is connected, \(A\) must lie in one component of \(X_A\). Letting \(x_0\) be any element of \(X_A\) and \(\gamma: [0,T]\to X\) be a shortest path from \(x_0\) to \(A\), the same reasoning as Proposition \ref{prop: ES basics}(i) shows that the image of \(\gamma\) is contained in \(X_A\). In particular, \(x_0\) is in the same component as \(A\). Therefore, \(X_A\) (and subsequently, \(X_B\)) is a single component. 

Applying Proposition \ref{prop: ES basics}(ii), we see that removing any point from \(E(A,B)\) allows for a path from \(X_A\) to \(X_B\). Thus, \(E(A,B)\) is minimal separating. 
\end{proof}

%%% DIRECTIONS TO THE EQUIDISTANT SET %%%

\subsection{Directions to the Equidistant Set}

Given an Alexandrov space \(X\), a compact set \(A\subseteq X\), and a point \(x\in X\setminus A\), we define \(\Theta_A\) to be the set of directions of shortest paths from \(x\) to \(A\): 
\[ \Theta_A = \{ [\gamma] \in S_x X : \gamma \text{ is a shortest path from } x \text{ to } A \} . \]

\begin{lem}\label{lem: disjoint directions} 
Let \(X\) be an Alexandrov space and let \(A\) and \(B\) be disjoint compact subsets of \(X\). For any \(x\in E(A,B)\), the sets \(\Theta_A\) and \(\Theta_B\) are disjoint compact subsets of \(S_xX\). 
\end{lem} 

\begin{proof} 
First, we verify that \(\Theta_A\) and \(\Theta_B\) are disjoint. 
By way of contradiction, suppose that there are paths \(\gamma_a\) and \(\gamma_b\) from \(x\) to \(a\in P_A(x)\) and \(b\in P_B(x)\), respectively, such that \(\angle(\gamma_a , \gamma_b) = 0\). 
By Corollary \ref{cor: subpath}, one of these paths is a subpath of the other. But since they have the same length, we must have \(\gamma_a = \gamma_b\), which is a contradiction since \(A\) and \(B\) are disjoint. 

To show that \(\Theta_A\) (and subsequently, \(\Theta_B\)) is compact, let \(\{[\gamma_n]\}_n\) be a sequence in \(\Theta_A\). For each \(n\) let \(\gamma_n: [0,d(x,A)] \to X\) be a shortest path in the equivalence class \([\gamma_n]\). Since \(L(\gamma_n) = d(x,A)\) for each \(n\), by Corollary \ref{cor: arzela} and the compactness of \(A\), \(\{\gamma_n\}_n\) contains a subsequence which converges uniformly to a shortest path \(\gamma: [0,d(x,A)] \to X\) which connects \(x\) to \(A\). Therefore, \(\{[\gamma_n]\}_n\) contains a subsequence which converges to \([\gamma]\in \Theta_A\). Thus, \(\Theta_A\) is compact. 
\end{proof}

%%% CONNECTIONS TO THE CUT LOCUS %%%

\subsection{Connections to the Cut Locus}

We close this section with a quick detour to establish some connections between equidistant sets and the cut locus. 
See \cite[\S \S 1,2]{ST1996} for further discussion on the cut locus for Alexandrov surfaces. 

\begin{defn} 
Let \(X\) be an Alexandrov space and \(K\subseteq X\) be compact. A point \(x\in X\) is a \emph{cut point} to \(K\) if there is a shortest path \(\gamma\) from \(x\) to \(K\) such that \(\gamma\) is not properly contained in any other shortest path to \(K\). The \emph{cut locus} to \(K\), denoted \(C(K)\), is the set of all cut points to \(K\). 
\end{defn} 

Let \(A\) and \(B\) are disjoint nonempty compact subsets of an Alexandrov space. Since there are no branching geodesics, and each \(x\in E(A,B)\) admits at least two shortest paths to \(A\cup B\), we immediately see that 
\[ E(A,B) \subseteq C(A\cup B) \] 
(this was similarly noted in \cite[p.~378]{zamfirescu}). 
This observation is beneficial as it allows us to apply much of the preliminary work done in \cite{ST1996} -- which was originally developed for the cut locus -- to equidistant sets (in particular, see Lemmas \ref{lem: basic lemma} and \ref{lem: rectifiable}).

%The purpose of this observation is it allows us to apply techniques for examining the cut locus to a compact set on Alexandrov surfaces, developed in \cite{ST1996}, to equidistant sets. 

%%%%%%%%%%%%%%%%%%%%%%%%%%%%%%%%%%%%%%%%%%%%%%%%%%%%%%%%%%

\section{Equidistant Sets on Alexandrov Surfaces} % EQUIDISTANT SETS CURVATURE BOUNDED BELOW %
\label{sec: ES CBB}

%%%%%%%%%%%%%%%%%%%%%%%%%%%%%%%%%%%%%%%%%%%%%%%%%%%%%%%%%%

We will now exclusively assume that \(X\) is a compact \(2\)-dimensional Alexandrov space, and \(A\) and \(B\) are nonempty disjoint closed subsets of \(X\). We additionally assume that \(X\) is without boundary, although the doubling theorem will allow us to extend Theorem \ref{thm: simplicial} to spaces with boundary.

The discussion in this section expands upon \cite[\S 4]{bv2006}, \cite[\S 1]{bv2007}, and \cite[\S 3]{herreros}, as we consider Alexandrov spaces instead of Riemannian manifolds, and disjoint nonempty compact sets instead of distinct points.

%%% SECTORS AND WEDGES %%%

\subsection{Sectors and Wedges}

The following definition gives a construction of `wedge-shaped' neighborhoods which we will use to describe the local behavior of \(E(A,B)\). We assume the radius \(\rho\) always satisfies \(\rho < \inf \{d(a,b) : a\in A , b\in B\}\), as well as being sufficiently small so that the open ball \(B_\rho(x)\) is homeomorphic to a disk, and the boundary 
\[ \partial B_\rho(x) = \{ y \in X : d(x,y) = \rho \} \] 
is homeomorphic to a circle. Given that \(A\) and \(B\) are disjoint and compact, and \(X\) is a \(2\)-manifold (Proposition \ref{prop: dimension}(iii)), such a \(\rho\) always exists.

\begin{defn}
Let \(X\) be an Alexandrov surface and \(K\subseteq X\) be nonempty and compact. A \emph{sector} of radius \(\rho\) at \(x \in C(K)\) is a component of 
\[ B_\rho(x) \setminus \{ \gamma(t) : [\gamma] \in \Theta_K , t\in [0,\rho) \} . \] 
Now let \(A\) and \(B\) be disjoint nonempty compact sets. We define a \emph{wedge} of radius \(\rho\) at \(x\in E(A,B) \subseteq C(A\cup B)\), which we will denote \(W_\rho(x)\), to be any sector at \(x\) such that one side is bounded by a shortest path to \(A\), and the other is bounded by a shortest path to \(B\). 
\end{defn}

%We further define the \emph{wedge boundary}, \(\partial W_{\rho}(x)\), to be the segment of \(\partial B_\rho(x)\) contained in the closure of \(W_\rho(x)\). 

See Figure \ref{fig: wedges} for clarification of sectors and wedges. Note that wedges and sectors are open sets. Furthermore, every point of the equidistant set admits at least two wedges (and always an even number), but only finitely many (see Lemma \ref{lem: finite wedges}).

\begin{figure}[h] % FIGURE
\begin{center}
\includegraphics[width=0.35\linewidth]{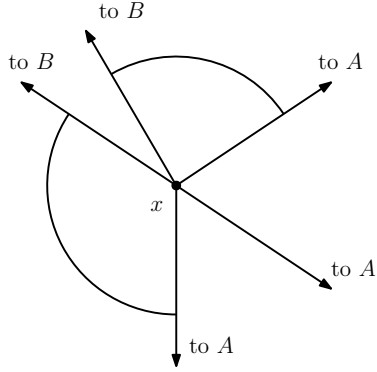}
\caption{A point \(x\in E(A,B)\) which admits five sectors, but only two wedges.} 
\label{fig: wedges} 
\end{center}
\end{figure}

\begin{lem}\label{lem: finite wedges} 
Let \(X\) be a compact Alexandrov surface and let \(A,B\subseteq X\) be disjoint closed sets. For any \(x\in E(A,B)\), there are only finitely many wedges at \(x\). 
\end{lem} 

\begin{proof}
By way of contradiction, suppose there are infinitely many wedges at \(x\in E(A,B)\). Then there must be infinitely many distinct directions \(\{\theta_{A,n}\}_{n=1}^\infty \subseteq \Theta_A\) and \(\{\theta_{B,n}\}_{n=1}^\infty \subseteq \Theta_B\) which form the wedges. Furthermore, we can assume that these sequences are alternating in the sense that \(\theta_{B,n}\) is always between \(\theta_{A,n}\) and \(\theta_{A,n+1}\) (i.e., if we start from \(\theta_{A,1}\) we can move along a geodesic in \(S_xX\) which passes through \(\theta_{B,1}\), then \(\theta_{A,2}\), then \(\theta_{B,2}\), and so on). Since \(\Theta_A\) is compact, there is a convergent subsequence 
\[ \theta_{A,n_k} \to \bar{\theta} \in \Theta_A . \] 
The subsequence \(\{\theta_{B,n_k}\}_k\) of directions between \(\theta_{A,n_k}\) and \(\theta_{A,n_k+1}\) also converges to \(\bar{\theta}\). Thus, \(\bar{\theta}\) is the direction of a shortest path to both \(A\) and \(B\), which contradicts Lemma \ref{lem: disjoint directions}. 
\end{proof}

Given that \(E(A,B) \subseteq C(A\cup B)\), and a wedge at \(x\in E(A,B)\) is simply a special case of a sector at \(x\in C(A\cup B)\), one can use the exact same proofs found in \cite{ST1996} for the following lemmas to arrive at Corollary \ref{cor: arc}.

\begin{lem}[{\cite[Basic Lemma]{ST1996}}]\label{lem: basic lemma} 
Let \(R_\rho(x)\) be a sector at \(x\in C(K)\). Then there exists a point \(y \in C(K) \cap R_\rho(x)\) and a Jordan arc \(\sigma: [0,1]\to C(K) \cap \overline{R_\rho(x)}\) such that \(\sigma(0) = x\) and \(\sigma(1) = y\). 
\end{lem}

\begin{lem}[{\cite[Lemma 2.3]{ST1996}}]\label{lem: rectifiable} 
Every Jordan arc \(\sigma: [0,1]\to C(K)\) constructed in Lemma \ref{lem: basic lemma} is rectifiable. 
\end{lem}

\begin{cor}\label{cor: arc}
Let \(W_{\rho}(x)\) be a wedge at \(x\in E(A,B)\). Then there exists a point \(y\in E(A,B) \cap W_\rho(x)\) and a rectifiable Jordan arc \(\sigma: [0,1]\to E(A,B) \cap \overline{W_\rho(x)}\) such that \(\sigma(0) = x\) and \(\sigma(1)=y\). 
\end{cor}

%%% DIRECTION OF EQUIDISTANT SET %%%

\subsection{The Direction of the Equidistant Set} 

Although the equidistant set is not generally a smooth curve, given any point in the equidistant set, we can at least predict the directions from the starting point which stay in the equidistant set. These directions are precisely the bisectors of any wedge at that point (Theorem \ref{thm: bisector}). To see why this is true, first recall our function \(f_{AB}\) defined by 
\[ f_{AB}(x) = d(x,A) - d(x,B) . \] 
Since \(f_{AB}\) is a linear combination of distance functions (to compact sets), it admits a one-sided directional derivative (see \cite{FOV} for a detailed exposition of this derivative). 

\begin{thm}\label{thm: one sided deriv}
Let \(X\) be an Alexandrov space, \(\gamma:[0,T]\to X\) a shortest path, and \(K\) a compact set not containing \(\gamma(0)\). Then 
\[ \lim_{t\to 0} \frac{d(\gamma(t),K) - d(\gamma(0),K)}{t} = -\cos(\angle_{\min}) \] 
where \(\angle_{\min}\) is the infimum of angles between \(\gamma\) and any shortest path connecting \(\gamma(0)\) to \(K\). 
\end{thm} 

Letting \(x\) be a point of \(E(A,B)\), and \(W_\rho(x)\) be a wedge at \(x\), if \(\gamma\) is a shortest path emanating from \(x\), Theorem \ref{thm: one sided deriv} gives us 
\[ \lim_{t\to 0^+} \frac{f_{AB}(\gamma(t)) - f_{AB}(\gamma(0))}{t} 
= -\cos\big( \angle(\gamma , \Theta_A) \big) + \cos\big( \angle(\gamma , \Theta_B) \big) . \] 
Since \(E(A,B)\) is the level set \(f_{AB} = 0\), the directional derivative in the equidistant set should be \(0\). Therefore, we expect the equidistant set to be locally well approximated by shortest paths satisfying 
\[\cos\big( \angle(\gamma , \Theta_A) \big) = \cos\big( \angle(\gamma , \Theta_B) \big) . \] 

Now let \(\gamma_A\) and \(\gamma_B\) be the shortest paths to \(A\) and \(B\), respectively, which bound our wedge \(W_\rho(x)\). Given that there are no other shortest paths from \(x\) to \(A\) or \(B\) in the wedge, 
\[ \angle(\gamma , \Theta_A) = \angle (\gamma, \Theta_B) \iff \angle(\gamma,\gamma_A) = \angle(\gamma , \gamma_B) . \] 
Thus, we expect the direction of the equidistant set to be exactly halfway between the directions which determine the wedge. 
Note however, there may not be an actual shortest path which realizes the desired direction, as the following example shows. 

\begin{ex}\label{ex: flat sphere} 
Let \(C\) be the compact Alexandrov surface obtained by doubling the closed unit disk. 
If \(a\) is the point at the center of one disk and \(b\) is the point at the center of the other disk, then \(E(a,b)\) is precisely the identified boundary of the disks. At any point of \(E(a,b)\), the space of directions is a unit circle, but there are no shortest paths which realize the two directions along the identified boundary. 
\end{ex} 

Although there may not be a shortest path which admits the desired direction, we can always achieve the desired direction as a limit of shortest paths. Given a sequence \(\{x_n\}_n \subseteq E(A,B) \cap W_\rho(x)\) such that \(x_n \to x\), if \(\gamma_n\) is the shortest path from \(x\) to \(x_n\), we expect to find 
\[ \lim_{n\to\infty} \Bigg( \lim_{t\to 0^+} \frac{f_{AB}\big(\gamma_n(t)\big) - f_{AB}(x)}{t} \Bigg) = 0 \]
since \(f_{AB}(x_n) = f_{AB}(x) = 0\). This idea is formalized by the Theorem \ref{thm: bisector}, but first we establish a preliminary lemma.

\begin{lem}\label{lem: desired direction} 
Let \(x\in E(A,B)\) be a wedge at \(x\), and let \(W_\rho(x)\) be a wedge at \(x\) bounded by shortest paths \(\gamma_A\) and \(\gamma_B\) to \(A\) and \(B\) respectively. If there exists a shortest path \(\eta: [0,T] \to \overline{W_\rho(x)}\) with \(\eta(0) = x\) such that the image of \(\eta\) intersects \(E(A,B)\) in an infinite sequence of distinct points converging to \(x\), then 
\[ \angle ( \eta , \gamma_A) = \angle (\eta , \gamma_B) = \begin{cases} \frac{1}{2} \angle (\gamma_A , \gamma_B) & \text{ if } \angle (\eta,\gamma_A) \leq \frac{1}{2} \diam S_xX \\ \diam S_xX - \frac{1}{2} \angle (\gamma_A , \gamma_B) & \text{ if } \angle (\eta,\gamma_A) > \frac{1}{2} \diam S_xX . \end{cases} \] 
\end{lem} 

\begin{rem} 
When \(\diam S_xX = \pi\), the angle formula matches our Euclidean intuition. 
\end{rem} 

\begin{proof} 
Let \(\{t_n\}_{n=1}^\infty\) be the sequence such that \(\eta(t_n) \in E(A,B)\) for all \(n\) and \(\eta(t_n) \to x\) (clearly \(t_n\to 0\)). 
By Theorem \ref{thm: one sided deriv}, 
\begin{align*} 
0 & = \lim_{n\to \infty} \frac{f_{AB} (\eta(t_n)) - f_{AB}(x)}{t_n} 
\\ & = \lim_{n\to \infty} \frac{d(\eta(t_n),A) - d(x,A)}{t_n} - \lim_{n\to \infty} \frac{d(\eta(t_n),B) - d(x,B)}{t_n} 
\\ & = \lim_{t\to 0^+} \frac{d(\eta(t),A) - d(\eta(0),A)}{t} - \lim_{t\to 0^+} \frac{d(\eta(t),B) - d(\eta(0),B)}{t}
\\ & = -\cos\big( \angle(\eta , \gamma_A) \big) + \cos\big( \angle( \eta,\gamma_B) \big) . 
\end{align*} 
Thus, \(\angle(\eta,\gamma_A) = \angle(\eta,\gamma_B)\), so \([\eta]\) is one of (up to) two directions in \(S_xX\) equidistant to \([\gamma_A]\) and \([\gamma_B]\). If the wedge is acute, then \(\angle(\eta,\gamma_A) = \frac{1}{2} \angle (\gamma_A , \gamma_B)\). Conversely, if the wedge is obtuse, then \(\angle(\eta,\gamma_A) = \diam S_xX - \frac{1}{2} \angle (\gamma_A , \gamma_B)\). 
\end{proof}

\begin{defn} 
Let \(\sigma:[0,T]\to X\) be a path in an Alexandrov space. If for every sequence \(t_n \to t^+\), the sequence of directions \(\{\gamma_n\}_n \subseteq S_xX\) from \(\sigma(t)\) to \(\sigma(t_n)\) converges, then we call the limiting direction the \emph{right tangent} at \(\sigma(t)\) and denote it \(\sigma^+(t)\). 
\end{defn}

The following theorem shows that the the equidistant set has a (one-sided) tangent direction and it is the bisector of the directions which determine the wedge. Compare to \cite[Theorem A]{herreros} or \cite[Lemma 2.1]{ST1996}.

\begin{thm}\label{thm: bisector} 
Let \(W_{\rho}(x)\) is a wedge at \(x \in E(A,B)\), bounded by shortest paths \(\gamma_A\) and \(\gamma_B\), to \(A\) and \(B\) respectively. For any curve 
\[ \sigma: [0,1] \to E(A,B) \cap \overline{W_{\rho}(x)} \] 
such that \(\sigma(0) = x\), the right tangent \(\sigma^+(0)\) exists and is the bisector of the directions \([\gamma_A]\) and \([\gamma_B]\). 
\end{thm}

\begin{proof} 
Let \(\theta^* \in S_xX\) be the direction which is the desired bisector of \([\gamma_A]\) and \([\gamma_B]\), the directions determining our wedge. 
For any given \(\varepsilon>0\), let \(B_\varepsilon(\theta^*)\) be the open interval of radius \(\varepsilon\) around \(\theta^* \in S_xX\), and define the set 
\[ C_{\pm \varepsilon} = \exp_x \Big( B_\varepsilon (\theta^*) \times (0,\delta) \Big) \] 
for some fixed \(\delta >0\). 
In order to prove the theorem, we will show that \(\sigma(s) \in C_{\pm \varepsilon}\) for all sufficiently small \(s>0\). 

Seeking a contradiction, we suppose the alternative possibilities: \\ 
(i) there is a constant \(c>0\) such that \(\sigma(s)\) is outside of \(C_{\pm \varepsilon}\) whenever \(0<s<c\); or \\ 
(ii) there is a sequence \(\{s_n\}_{n=1}^\infty\) with \(s_n\to 0^+\) such that infinitely many \(\sigma(s_n)\) are in \(C_{\pm \varepsilon}\) and infinitely many \(\sigma(s_n)\) are outside of \(C_{\pm \varepsilon}\). 

If case (i), then without loss of generality, we can assume that \(d(z,A)>d(z,B)\) for all \(z\) in \(C_{\pm\varepsilon}\). Letting \(\gamma\) be a shortest path with \([\gamma]\in B_{\varepsilon}(\theta^*)\) and \(\angle ( \gamma,\gamma_A ) < \angle (\gamma,\gamma_B)\), we have 
\begin{align*} 
0 & = \lim_{t\to 0^+} \frac{f_{AB}(\sigma(t)) - f_{AB}(x)}{d(\sigma(t),x)} 
\\ & \leq \lim_{t\to 0^+} \frac{f_{AB}(\gamma(t)) - f_{AB}(x)}{t} 
\\ & = -\cos\big( \angle(\gamma,\gamma_A) \big) + \cos\big( \angle(\gamma,\gamma_B) \big) 
\\ & < 0 
\end{align*} 
which is a contradiction. 

If case (ii), then we can find a shortest path \(\gamma\) with \(0 < \angle (\gamma , \theta^*) \leq \varepsilon\) such that \(\sigma\) intersects \(\gamma\) infinitely many times in a sequence which converges to \(x\). However, this contradicts Lemma \ref{lem: desired direction}. 

We conclude that \(\sigma(s)\in C_{\pm \varepsilon}\) for all \(s\) sufficiently small. Since \(\varepsilon\) can be taken arbitrarily close to \(0\), we find \(\sigma^+(0)\) exists and is the bisector direction \(\theta^*\). 
\end{proof}

\subsection{The Equidistant Set is a Simplicial 1-Complex} % SIMPLICIAL COMPLEX

For sufficiently small radius, a wedge will only contain a single Jordan arc segment of the equidistant set. 

\begin{lem}\label{lem: single arc} 
Let \(x\in E(A,B)\) be given and let \(W_\rho(x)\) be a wedge at \(x\). There is a sufficiently small \(\delta\) such that the intersection of \(E(A,B)\) and \(\overline{W_\varepsilon(x)}\) is a Jordan arc for all \(0<\varepsilon < \delta\). 
\end{lem} 

\begin{proof} 
By Corollary \ref{cor: arc}, there is a Jordan arc \(\sigma:[0,1]\to E(A,B) \cap \overline{W_\rho(x)}\) with \(\sigma(0) = x\). We assume that \(\rho\) is taken sufficiently small so that \(d(\sigma(0),\sigma(1))=\rho\). By way of contradiction, suppose that \(W_{1/n}(x)\) contains a point \(x_n\in E(A,B) \setminus \sigma([0,1])\) for all \(n\). Passing to a subsequence if necessary, we can assume that \(\{x_n\}_{n\geq 1/\rho}\) is contained in one of the two components of \(W_\rho(x) \setminus \sigma([0,1])\), so we will assume it is in the component whose boundary contains \(\gamma_A\), a shortest path from \(x\) to \(A\). For each \(n\), let \(\eta_n\) be a shortest path from \(x_n\) to \(B\). Since \(x_n\to x\), we must have a subsequence of \(\{\eta_n\}\) which converges uniformly to a shortest path from \(x\) to \(B\). However, this is impossible since \(\eta_n\) cannot cross \(\gamma_A\) or \(\sigma\), and \(W_{1/n}(x)\) contains no shortest paths from \(x\) to \(B\). 
\end{proof}

\begin{cor} 
There are only finitely many \(x\in E(A,B)\) which admit more than two wedges. 
\end{cor} 

\begin{proof} 
If there was an infinite sequence of distinct points \(\{x_n\}_{n=1}^\infty \subseteq E(A,B)\) which admit more than two wedges, then by the compactness of \(E(A,B)\), there is a subsequence which converges to some \(\bar{x}\in E(A,B)\). By Lemma \ref{lem: finite wedges} there are only finitely many wedges at \(\bar{x}\), so there is a wedge \(W_\rho(\bar{x})\) containing infinitely many of these points for any \(\rho>0\), which contradicts Lemma \ref{lem: single arc}.  
\end{proof}

\begin{thm}\label{thm: simplicial} 
Let \(X\) be a compact Alexandrov surface (possibly with boundary). For any pair of disjoint nonempty compact subsets \(A,B\subseteq X\), the equidistant set \(E(A,B)\) is homeomorphic to a finite closed simplicial 1-complex. 
\end{thm}

\begin{proof} 
First, suppose that \(X\) has no boundary. At each \(x\in E(A,B)\) which admits more than two wedges (of which there are finitely many), we place a vertex at \(x\). By compactness of \(E(A,B)\) and Lemma \ref{lem: single arc}, there is a natural concatenation of finitely many Jordan arcs, which forms an edge between appropriate vertices. In the case that a point \(x\) belongs to a component of \(E(A,B)\) that is a simple closed curve (so there are no points with more than two wedges) we simply place a vertex at \(x\) and then create our edge of Jordan arcs as before. Thus, \(E(A,B)\) is homeomorphic to a finite simplicial 1-complex. 

In the case that \(X\) has nonempty boundary, we apply the doubling theorem to get a space without boundary. By our discussion in the previous paragraph, \(E(A,B)\) is a finite simplicial 1-complex in this space. 
By compactness, the boundary \(\partial X\) can only intersect \(E(A,B)\) finitely many times. 
To see this, first note that \(\partial X \cap E(A,B)\) cannot contain a Jordan arc: If there is a wedge which is bisected by the boundary (Theorem \ref{thm: bisector}), then the directions that determine the wedge, \(\gamma_A\) and \(\gamma_B\), must terminate at the same point, which is a contradiction. 
Therefore, the intersection happens at Jordan arcs in \(E(A,B)\) passing through the boundary. If we assume there are infinitely many such intersections, then there must be a cluster point, but there can be no wedge at such a cluster point. Thus, removing the part of \(E(A,B)\) contained in the `doubled' part of the space and placing vertices at the finitely many points in \(\partial X\) leaves us with a finite simplicial 1-complex. 
\end{proof}

We can conclude from \cite[Theorem A]{ST1996} that Hausdorff dimension of equidistant sets (with the conditions of the above theorem) is always 1. 
However, it is possible to produce a cut locus with infinite 1-dimensional measure on a compact surface (see \cite[Theorem B]{itoh} or \cite[Example 4]{ST1996}). Equidistant sets on compact surfaces, on the other hand, will always have finite measure.

\begin{cor}\label{cor: finite measure} 
Let \(X\) be a compact Alexandrov surface. If \(A,B\subseteq X\) are nonempty disjoint and compact, then the \(1\)-dimensional Hausdorff measure of \(E(A,B)\) is nonzero and finite. 
\end{cor}

\begin{proof} 
Let \(\{W_i\}_{i\in I}\) be an open cover of \(E(A,B)\) by wedges which admit a unique rectifiable Jordan arc in their interior (see Lemma \ref{lem: single arc}). 
Compactness of \(E(A,B)\) implies we can find a finite subcover \(\{W_k\}_{k=1}^n\). 
Let \(J_k\) be the unique rectifiable segment of \(E(A,B)\) contained in the closure of each wedge \(W_k\) of this finite subcover. 
Then we have 
\[ 0 < \scrH^1(J_1) \leq \scrH^1(E(A,B)) \leq \sum_{k=1}^n \scrH^1(J_k) < \infty . \qedhere \] 
\end{proof}

%%% HOMOLOGY %%%

\subsection{A Bound on the Homology of the Equidistant Set} 

Now that we know the equidistant set is a \(1\)-complex, it is appropriate to ask if we can provide any further details on the structure. 
While the specifics of the equidistant set will depend greatly on the structure and placement of the sets \(A\) and \(B\), we can at least provide an upper bound on the dimension of the first homology group (Theorem \ref{thm: homology}). This bound depends on both the genus of the surface as well as the number of disconnected pieces of \(A\) and \(B\). 

The following theorem is a generalization of \cite[Theorem 4.2]{bv2006} for \(n=2\). 

\begin{thm}\label{thm: homology} 
Let \(X\) be a compact Alexandrov surface (without boundary). If \(A,B\subseteq X\) are nonempty disjoint and compact, then \(H_1(E(A,B)) = \Z^k\) for some positive integer \(k\) satisfying 
\[ 1 \leq k \leq \dim H_1 (X; \Z_2) + \dim H_0(A) + \dim H_0(B) - 1 . \] 
\end{thm} 

\begin{rem} 
Note that the \(k\) in the statement of Theorem \ref{thm: homology} is in fact finite, even though \(\dim H_0(A)+\dim H_0(B)\) could certainly be infinite. 
\end{rem} 

\begin{proof} 
For simplicity of notation, let \(E\) be the equidistant set \(E(A,B)\). 
Since the equidistant set is a simplicial \(1\)-complex (Theorem \ref{thm: simplicial}), \(\dim H_1(E) = \dim H_1 (E;\Z_2)\), so we will use \(\Z_2\) coefficients to avoid issues of orientability. 
Consider the following portion of the long exact sequence for the pair \((X,E)\): 
\[\begin{tikzcd}
\cdots \arrow[r, "\partial_{3}"] & H_{2}(E;\Z_2) \arrow[r, "i_{2}"] & H_2(X;\Z_2) \arrow[r, "p_{2}"] & H_{2}(X,E; \Z_2) 
\\ \arrow[r, "\partial_2"] & H_{1}(E;\Z_2) \arrow[r, "i_{1}"] & H_{1}(X;\Z_2) \arrow[r, "p_{1}"] & \cdots . 
\end{tikzcd} \]
First, we observe that \(E\) has empty interior (Corollary \ref{cor: no interior}), so \(H_{2}(E;\Z_2) = 0\). Furthermore, \(X\) is a manifold without boundary (Proposition \ref{prop: dimension}(iii)), so \(H_{2}(X;\Z_2) = \Z_2\). We claim that \(H_2(X,E;\Z_2) = (\Z_2)^\ell\) where \(\ell\) is a positive integer satisfying 
\begin{equation}\label{eq: dim inequality} 
2 \leq \ell \leq \dim H_0(A; \Z_2) + \dim H_0(B;\Z_2) . 
\end{equation}

Before proving the claim, we show how it proves the theorem. 
Filling in the known homology groups, our portion of the long exact sequence becomes 
\[\begin{tikzcd} 
\cdots \arrow[r, "\partial_3"] & 0 \arrow[r, "i_{2}"] & \Z_2 \arrow[r, "p_{2}"] & ( \Z_2 )^\ell \arrow[r, "\partial_2"] & H_{1}(E;\Z_2) \arrow[r, "i_{1}"] & H_{1}(X;\Z_2) \arrow[r, "p_{1}"] & \cdots . 
\end{tikzcd} \]
Exactness of this sequence tells us that \(\ker(\partial_2) = \im(p_2) = \Z_2\), so we must have \(\im(\partial_2) = (\Z_2)^{\ell-1}\). Given that \(\ell\geq 2\), this gives us the lower bound for \(\dim H_1(E;\Z_2)\). 
To get the upper bound, notice that \(\ker(i_1) = \im(\partial_2) = (\Z_2)^{\ell-1}\), and \(\im(i_1)\) is contained in \(H_1(X;\Z_2)\), so we have \(\dim H_1(E;\Z_2) \leq \dim H_1(X;\Z_2) + \ell - 1\).

To prove \(H_2(X,E;\Z_2) = (\Z_2)^\ell\) with \(\ell\) as in \eqref{eq: dim inequality}, we define the open sets 
\[ X_A = \{x\in X : d(x,A) < d(x,B) \} \quand X_B = \{x\in X : d(x,B) < d(x,A)\} . \] 
Let \(\ell_A\) and \(\ell_B\) be the number of components of \(X_A\) and \(X_B\) respectively. 
First, we will show that \(\ell_A\) (and subsequently, \(\ell_B\)) is finite. Let 
\[\varepsilon = \inf \{ d(a,b) : a\in A , b\in B\} . \] 
Since \(A\) and \(B\) are compact and disjoint, \(\varepsilon >0\). By compactness of \(A\), we can find finitely many \(a_i \in A\) such that \(A\subseteq \bigcup_{i=1}^N B_{\varepsilon/3}(a_i)\). Furthermore, by the definition of \(\varepsilon\), 
\[  \bigcup_{i=1}^N B_{\varepsilon/3}(a_i) \subseteq X_A . \] 
By Proposition \ref{prop: ES basics}(ii), \(X_A\) can have no more than \(N\) components (one for each \(a_i\)). Thus \(\ell_A\) is finite. To be more precise, one can use the same reasoning as Corollary \ref{cor: minimal sep} to see that if \(B_{\varepsilon/3}(a_i) \cap B_{\varepsilon/3}(a_j)\) is nonempty, then these balls lie in the same component of \(X_A\). 
It follows that \(1 \leq \ell_A \leq \dim H_0(A)\). 
Applying the same reasoning for \(\ell_B\), we see that \(X \setminus E\) has \(\ell_A + \ell_B\) components. We conclude, via Lefschetz duality, that  
\[ \dim H_2(X,E;\Z_2) = \dim H^0(X\setminus E;\Z_2) = \ell_A + \ell_B \] 
where \(2 \leq \ell_A + \ell_B \leq \dim H_0(A) + \dim H_0(B)\). 
\end{proof} 

If we additionally assume that \(A\) and \(B\) are connected, then Corollary \ref{cor: minimal sep} gives the following corollary to Theorem \ref{thm: homology}. 
This bound for the homology is precisely that found in \cite[Theorem 4.2]{bv2006} (for \(n=2\)) when \(A\) and \(B\) are singleton sets. 

\begin{cor} 
Let \(X\) be a compact Alexandrov surface (without boundary). If \(A,B\subseteq X\) are nonempty disjoint compact and connected, then 
\[ 1 \leq \dim H_1 ( E(A,B) ) \leq \dim H_1 (X; \Z_2) + 1 . \] 
\end{cor}

%\begin{thm}\label{thm: n homology} %%%%% IS THIS WORTH INCLUDING %%%%%
%Let \(X\) be a compact \(n\)-dimensional Alexandrov space, which is also a topological \(n\)-manifold. If \(A,B\subseteq X\) are nonempty disjoint and compact, then 
%\[ 1 \leq \dim H_{n-1} (E(A,B);\Z_2) \leq \dim H_{n-1}(X; \Z_2) + \dim H_0(A;\Z_2) + \dim H_0(B;\Z_2) - 1 . \] 
%\end{thm} %%%%%%%%%%%%%%%%%%%%%%%%%%%%%%%%%%%%%%%%%

%%%%%%%%%%%%%%%%%%%%%%%%%%%%%%%%

\section{Equidistant Sets in the Plane} %%% IN THE PLANE %%%
\label{sec: the plane} 

%%%%%%%%%%%%%%%%%%%%%%%%%%%%%%%%%

We now use the techniques developed in previous sections to answer two prompts for equidistant sets in the plane given in \cite{ponce}.

\subsection{Equidistant Sets Determined by Connected Sets} %%% CONNECTED SETS %%%

First, we address the following problem: 
\begin{quote}
\cite[p.~32]{ponce} \textit{Problem}: Characterize all closed sets of \(\R^2\) that can be realized as the equidistant set of two connected disjoint closed sets. 
\end{quote}

%Some effort towards a partial solution is given in \cite{VVOF2018}, where the authors give explicit parameterizations for equidistant sets determined by a line and the epigraph of a \(C^2\) convex function. In this case, it seems the goal is to give explicit formulas for the equidistant set. However, if we look at the problem from a topological perspective, we see that every equidistant set is homeomorphic to either a line or a circle. This is a generalization of \cite[Theorem 3.2]{loveland} for dimension \(n=2\). 

Interestingly, the answer to this question has been known since 1975: In \cite[Theorem 1]{bell} it is shown that the equidistant sets determined by two nonempty disjoint closed connected subsets of the Euclidean plane is a connected topological 1-manifold (i.e. the equidistant set is homeomorphic to a circle or the real line). 
However, the result is misquoted in \cite{loveland} (and subsequently in \cite{ponce}) as being true only for \emph{compact} subsets.\footnote{This is likely due to Bell's own abstract, which claims the result for ``mutually disjoint plane continua'' (a continuum being generally understood to be a connected compact set), but the statement of the theorem, and its proof, is for connected closed sets.} 
In order to illustrate the strength of some of the techniques developed in Section \ref{sec: ES CBB}, we give a quick alternative proof of Bell's theorem.

\begin{thm}\label{thm: line or circle} 
If \(A\) and \(B\) are nonempty disjoint closed connected subsets of the Euclidean plane, then \(E(A,B)\) is a connected topological \(1\)-manifold. 
\end{thm} 

\begin{proof} 
First, the fact that \(E(A,B)\) is connected is essentially a consequence of Corollary \ref{cor: minimal sep}. If \(E(A,B)\) is not connected, then it has at least two connected components. As such, we could find one connected component of \(E(A,B)\) which separates the plane, but this would be a proper subset of \(E(A,B)\), so it is not minimal separating. See also \cite[Theorem 4]{wilker} for an alternative proof of connectedness. 

Now, let \(x\in E(A,B)\) be given. In order to prove the theorem, we will show that for some \(\varepsilon>0\), \(E(A,B) \cap B_\varepsilon(x)\) is the interior of a Jordan arc, and therefore every point of \(E(A,B)\) is contained in a neighborhood homeomorphic to an open interval of the real line. 

Fix a real number \(r > d(x,A)\), and note that \(\overline{B}_{3r}(x)\) is a compact Alexandrov surface (with boundary) of curvature \(\geq 0\). 
We define the compact sets 
\[A_x = A\cap \overline{B}_{3r}(x) \quand B_x = B \cap \overline{B}_{3r}(x) \] 
and consider \(E(A_x,B_x)\) in the space \(\overline{B}_{3r}(x)\). 
Note that for every \(y\in \overline{B}_r(x) \cap E(A,B)\), if \(a_y\in P_A(y)\), then 
\[ d(x,a_y) \leq d(x,y) + d(y,A) \leq 2d(x,y) + d(x,A)  < 3 r \] 
so \(P_A(y) \subseteq \overline{B}_{3r}(x)\) (and the same is true for \(P_B(y)\)). As such, 
\[ E(A_x,B_x) \cap \overline{B}_r(x) = E(A,B) \cap \overline{B}_r(x) . \] 
Given that \(A\) and \(B\) are each connected, every \(x\in E(A,B)\) admits exactly two wedges (if there were more than two wedges at a point, we could use the set \(A\) and two distinct paths from \(x\) to \(A\) to separate \(B\) from itself), so let \(W^1_\rho(x)\) and \(W^2_\rho(x)\) be two such wedges at \(x\) (for some sufficiently small \(\rho>0\)).  
Applying Lemma \ref{lem: single arc} shows that we can find an \(\varepsilon\) with \(0<\varepsilon<\rho\) such that  
\[E(A,B) \cap \overline{B}_\varepsilon(x) = \big( E(A,B) \cap \overline{W^1_\varepsilon(x)} \big) \cup \big( E(A,B) \cap \overline{W^2_\varepsilon(x)} \big) \] 
is the image of a Jordan arc. 
\end{proof}

\subsection{Hausdorff Dimension of the Equidistant Set} % DIMENSION 

We now answer the following question (the term \emph{focal sets} refers to the sets \(A\) and \(B\) which determine the equidistant set):

\begin{quote}
\cite[p.~32]{ponce} \textit{Question}: Does there exist an equidistant set in the plane with connected disjoint focal sets, having Hausdorff dimension greater than 1? What about other notions of dimension? How does the dimension of an equidistant set depend on the dimension of its focal sets? 
\end{quote}

Note that the question does not require the \emph{closure} of the sets be disjoint. This allows for interesting constructions, such as the example in \cite{wilker}, which shows an equidistant set (with disjoint focal sets) can have Hausdorff dimension 2. This is achieved by using two `interlocking combs' as the sets \(A\) and \(B\). In fact, as the following example illustrates, we can exploit overlapping closures to produce equidistant sets in the plane of any Hausdorff dimension between \(1\) and \(2\). 

\begin{ex} 
Let \(K\) be the (boundary of the) Koch snowflake. %(see Figure \ref{fig: koch snowflake}). % INCLUDE A FIGURE ?
Define \(A\) as the bounded subset of \(\R^2 \setminus K\) and \(B\) as the unbounded subset of \(\R^2 \setminus K\). Then \(A\cap B = \varnothing\), but \(E(A,B)\) is precisely \(K\), which has Hausdorff dimension \(\log_3(4) \approx 1.26186\). 
\end{ex} 

On the other hand, using two sets with overlapping closure is `cheating' in some sense when working with equidistant sets, as we can have \(d(x,A) = d(x,B) = 0\) even when \(x\) is outside of \(A\) and \(B\). For this reason, it is more appropriate to ask the question for sets with disjoint closure. In this case, we find the equidistant set is always 1-dimensional, even if \(A\) and \(B\) are not connected.

\begin{thm}\label{thm: hausdorff dim} 
If \(A\) and \(B\) are nonempty disjoint closed subsets of the Euclidean plane, then the Hausdorff dimension of \(E(A,B)\) is \(1\). 
\end{thm} 

\begin{proof}
By the countable stability of the Hausdorff dimension \cite[p. 29]{falconer2}, 
\[ \dim_\scrH \bigcup_{n=1}^\infty F_n = \sup_{1\leq n < \infty} \dim_\scrH F_n \] 
so it is sufficient to show that the equidistant set is the union of countably many \(1\)-dimensional sets. 
Letting \(x\in E(A,B)\) and \(n\geq 1\) be given, we will show that \(\overline{B}_n(x) \cap E(A,B)\) is \(1\)-dimensional, which implies that  
\begin{align*} 
\dim_\scrH E(A,B) & = \dim_\scrH \bigcup_{n=1}^\infty \Big( \overline{B}_n (x) \cap E(A,B) \Big) 
\\ & = \sup_{1\leq n < \infty} \dim_\scrH \overline{B}_n (x) \cap E(A,B) 
\\ & = 1 . 
\end{align*}

First, following the same reasoning as in Theorem \ref{thm: line or circle}, we will consider the compact subspace \(\overline{B}_{R}(x)\), where 
\(R > 2n+d(x,A)\), 
so that the equidistant set determined by the disjoint compact sets  
\[A_x = A\cap \overline{B}_{R}(x) \quand B_x = B \cap \overline{B}_{R}(x) \] 
satisfies \(E(A,B) \cap \overline{B}_n(x) = E(A_x,B_x) \cap \overline{B}_n(x)\). 
It follows from Corollary \ref{cor: finite measure} that \(E(A_x,B_x)\) has Hausdorff dimension \(1\) and finite measure. 
Given that \(E(A,B) \cap \overline{B}_n(x)\) contains some rectifiable curve, 
\[ 0 < \mathscr{H}^1 \big( E(A,B) \cap \overline{B}_n(x) \big) \leq \mathscr{H}^1 \big( E(A_x, B_x ) \big) \] 
so \(E(A,B) \cap \overline{B}_n(x)\) has Hausdorff dimension \(1\). 
\end{proof}

%\end{onehalfspacing}

\bibliography{metricgeometry.bib} % BIB % 

\newcommand{\noopsort}[1]{} \newcommand{\printfirst}[2]{#1}
  \newcommand{\singleletter}[1]{#1} \newcommand{\switchargs}[2]{#2#1}
\begin{thebibliography}{VPRS00}

\bibitem[BBI01]{bbi2001}
D.~Burago, Yu. Burago, and S.~Ivanov.
\newblock {\em A Course in Metric Geometry}.
\newblock Graduate Studies in Mathematics. American Math. Soc., 2001.

\bibitem[Bel75]{bell}
H.~Bell.
\newblock Some topological extensions of plane geometry.
\newblock {\em Revista Colombiana de Matem{\'a}ticas}, 9:125--153, 1975.

\bibitem[BGP92]{bgp92}
Yu. Burago, M.~Gromov, and G.~Perel'man.
\newblock {A.D. Alexandrov} spaces with curvature bounded below.
\newblock {\em Russian Math. Surveys}, 47(2):1--58, 1992.

\bibitem[BV07]{bv2007}
J.~Bernhard and J.J.P. Veerman.
\newblock The topology of surface mediatrices.
\newblock {\em Topology and its Applications}, 154(1):54--68, 2007.

\bibitem[Fal90]{falconer2}
Kenneth~J. Falconer.
\newblock {\em Fractal Geometry: Mathematical Foundations and Applications}.
\newblock John Wiley \& Sons, 1990.

\bibitem[FOV21]{FOV}
Logan~S. Fox, Peter Oberly, and J.J.P. Veerman.
\newblock One-sided derivative of distance to a compact set.
\newblock {\em Rocky Mountain J. Math.}, 51(2):491--508, 2021.

\bibitem[GP93]{GP1993}
Karsten Grove and Peter Petersen.
\newblock A radius sphere theorem.
\newblock {\em Inventiones mathematicae}, 112:577--583, 1993.

\bibitem[HPV17]{herreros}
Pilar Herreros, Mario Ponce, and J.J.P. Veerman.
\newblock Equators have at most countably many singularities with bounded total
  angle.
\newblock {\em Annales Academiae Scientiarum Fennicae}, 42:837--845, 2017.

\bibitem[Ito96]{itoh}
Jin-Ichi Itoh.
\newblock The length of a cut locus on a surface and {Ambrose's} problem.
\newblock {\em Journal of Differential Geometry}, 43:642--651, 1996.

\bibitem[Lov76]{loveland}
L.D. Loveland.
\newblock When midsets are manifolds.
\newblock {\em Proceedings of the American Mathematical Society},
  61(2):353--360, 1976.

\bibitem[Per91]{perelman}
G.~Perelman.
\newblock {Alexandrov's} spaces with curvatures bounded from below {II}.
\newblock {\em (preprint) Leningrad Branch of Steklov Institute, St.
  Petersburg}, 1991.

\bibitem[Pet97]{petrunin}
Anton Petrunin.
\newblock Applications of quasigeodesics and gradient curves.
\newblock In {\em Comparison Geometry}, volume~30 of {\em MSRI Publications},
  pages 203--219. Cambridge University Press, 1997.

\bibitem[Pla02]{plaut}
Conrad Plaut.
\newblock Metric spaces of curvature \(\geq k\).
\newblock In {\em Handbook of Geometric Topology}, chapter~16, pages 819--898.
  Elsevier, 2002.

\bibitem[PS14]{ponce}
Mario Ponce and Patricio {Santib\'a\~nez}.
\newblock On equidistant sets and generalized conics: The old and the new.
\newblock {\em The American Mathematical Monthly}, 121(1):18--32, 2014.

\bibitem[Shi93]{shiohama}
Katsuhiro Shiohama.
\newblock {\em An Introduction to the Geometry of Alexandrov Spaces}, volume~8
  of {\em Lecture Notes Series}.
\newblock Seoul National University, 1993.

\bibitem[ST96]{ST1996}
Katsuhiro Shiohama and Minoru Tanaka.
\newblock Cut loci and distance spheres on {Alexandrov} surfaces.
\newblock In Arthur~L. Besse, editor, {\em Actes de la Table Ronde de
  G{\'e}om{\'e}trie Diff{\'e}rentielle en l'honneur de Marcel Berger}, number~1
  in S{\'e}minaires \& Congr{\`e}s, pages 531--559. Soci{\'e}t{\'e}
  Math{\'e}matique de France, 1996.

\bibitem[VB06]{bv2006}
J.J.P. Veerman and J.~Bernhard.
\newblock Minimally separating sets, mediatrices, and {Brillouin} spaces.
\newblock {\em Topology and its Applications}, 153(9):1421--1433, 2006.

\bibitem[VPRS00]{vprs}
J.J.P. Veerman, M.M. Peixoto, A.C. Rocha, and S.~Sutherland.
\newblock On {Brillouin} zones.
\newblock {\em Communications in Mathematical Physics}, 212(3):725--744, 2000.

\bibitem[Wil75]{wilker}
J.B. Wilker.
\newblock Equidistant sets and their connectivity properties.
\newblock {\em Proceedings of the American Mathematical Society},
  47(2):446--452, 1975.

\bibitem[Zam04]{zamfirescu}
Tudor Zamfirescu.
\newblock The cut locus in alexandrov spaces and applications to convex
  surfaces.
\newblock {\em Pacific Journal of Mathematics}, 217(2):375--386, 2004.

\end{thebibliography}
\bibliographystyle{alpha}

\end{document}